\documentclass[11pt,twoside]{article}

\usepackage{mathrsfs,amsfonts,amsmath,amssymb}
\usepackage{latexsym}
\newtheorem{satz}{Theorem}
\newtheorem{proposition}[satz]{Proposition}
\newtheorem{theorem}[satz]{Theorem}
\newtheorem{lemma}[satz]{Lemma}
\newtheorem{definition}[satz]{Definition}
\newtheorem{corollary}[satz]{Corollary}
\newtheorem{remark}[satz]{Remark}

\def\eps{\varepsilon}
\def\_phi{\varphi}

\def\a{\alpha}

\def\m{\times}

\def\C{{\mathbb C}}
\def\R{{\mathbb R}}
\def\E{\mathsf {E}}

\def\Z_N{{\mathbb Z}_N}
\def\Z{{\mathbb Z}}

\def\Gr{{\mathbf G}}

\def\D{{\mathbb D}}

\def\c{\circ}
\def\D{\Delta}

\author{Shkredov I.D.}
\title{ On a question of A. Balog
\footnote{
This work was supported by grant
Russian Scientific Foundation RSF 14--11--00433.}
}
\date{}
\begin{document}
\maketitle

\begin{center}
 Annotation.
\end{center}

{\it \small
    We give a partial answer to a conjecture of A. Balog, concerning the size of $AA+A$,
    where $A$ is a finite subset of real numbers.
    Also, we prove several new results on the cardinality of $A:A+A$, $AA+AA$ and $A:A + A:A$.
}
\\
%\\
%\\

\section{Introduction}
\label{sec:introduction}

Let $A\subset \R$ be  a finite set.
Define the \textit{sumset}, and respectively the \textit{product set}, by
$$A+A:=\{a+b:a,b\in{A}\}$$
and
$$AA:=\{ab:a,b\in{A}\}.$$
The Erd\H{o}s-Szemer\'{e}di \cite{ES} conjecture states, for all $\epsilon>0$,
$$\max{\{|A+A|,|AA|\}}\gg{|A|^{2-\epsilon}} \,.$$
Loosely speaking, the conjecture says that any set of reals (or integers) cannot be highly structured both in a multiplicative and additive sense.
The best result in the direction is due to Solymosi \cite{soly}.

\begin{theorem}
    Let $A\subset \R$ be a set.
    Then
$$
    \max\{ |A+A|, |AA| \} \gg |A|^{\frac{4}{3}} \log^{-\frac{1}{3}} |A| \,.
$$
\label{t:Solymosi_intr}
\end{theorem}

If one consider the set
$$
    AA+A = \{ ab+c ~:~ a,b,c\in A \}
$$
then the Erd\H{o}s-Szemer\'{e}di conjecture implies that $AA+A$ has size at least $|A|^{2-\eps}$
(we  assume for simplicity that $1\in A$).
In \cite{balog} Balog
%conjectured that
formulated a weaker hypothesis that
for all $\eps>0$ one has
$$
    |AA+A| \gg |A|^{2-\eps} \,.
$$
In the paper he proved the following result, which implies, in particular,
% that
 $|A A+A|\gg |A|^{3/2}$ and $|A A+A A|\gg |A||A/A|^{1/2}$.

\begin{theorem}
For every finite sets of reals  $A,B,C,D \subset \R$, we have
\begin{equation}\label{f:Balog1}
    |A C+ A||B C+ B|\gg |A||B||C| \,,
\end{equation}
and
\begin{equation}\label{f:Balog2}
    |A C+ A D||B C+ B D|\gg |B/A||C||D| \,.
\end{equation}
More precisely, see  \cite{SS1}
$$|(A\times B) \cdot \D(C)+ A\times B|\gg |A||B||C|\,$$
and
$$|(A\times B) \cdot \D(C)+ (A\times B) \cdot \D(D)|\gg |B/A||C||D|\,,$$
where
$$
    \D (A) := \{ (a,a) ~:~ a \in A \} \,.
$$
\label{t:Balog}
\end{theorem}

In \cite{B_RN_S} the authors have obtained a partial answer to a "dual"\, question
on the size of $A(A+A)$.
The main result of the paper is the following new bound for $A:A+A$, $AA+A$, more precisely, see Theorem \ref{t:main_balog} below.

\begin{theorem}
    Let $A\subset \R$ be a set.
    Then there is $\eps_1>0$ such that
\begin{equation}\label{f:main_balog1_intr}
    | A:A + A | \gg |A|^{3/2+\eps_1} \,.
\end{equation}
    Moreover, there is $\eps_2>0$ with
\begin{equation}\label{f:main_balog2_intr}
    |AA + A| \gg |A|^{3/2+\eps_2} \,,
\end{equation}
    provided by $|A:A| \ll |AA|$.
\label{t:main_balog_intr}
\end{theorem}

Also, we prove several results on the cardinality of $AA+AA$ and $A:A + A:A$,
see Theorem \ref{t:AA+AA} and Proposition \ref{p:AA+AA_Solymosi} below.
%The simplest  consequence of the results is the bound
%$$|AA+AA| \gg |A|^{5/3-\eps} \,,$$
%where $\eps>0$ is an arbitrary.

In paper \cite{RN_Z} Roche--Newton and Zhelezov conjectured there exist absolute constants $c$, $c'$ such that for any finite $A\subset \C$
%one has
the following holds
$$
    \left| \frac{A+A}{A+A} \right| \le c|A|^2 \implies |A+A| \le c' |A| \,.
$$
Similar conjectures  were made for the sets $\frac{A-A}{A-A}$, $(A-A)(A-A)$, $A(A+A+A+A)$ and so on.
We finish the paper giving a partial answer to a variant of the  conjecture of Roche--Newton and Zhelezov
$$
    |(A+A)(A+A) + (A+A)(A+A)| \ll |A|^2 \implies |A \pm A| \ll |A| \log^{} |A| \,,
$$
see Corollary \ref{c:RN_Z}.

The main idea of the proof is
%rather simple.
the following.
We need to estimate from below the sumset of two sets $A$ and $A:A$, say.
As in many problems of the type usual applications of Szemer\'{e}di--Trotter's theorem \cite{TV} or Solymosi's method \cite{balog} give us a lower bound of the form  $|A:A+A| \gg |A|^{3/2}$.
In paper \cite{ss2} the exponent $3/2$ was improved in the particular case of sumsets of convex sets.
After that the method was developed by several authors, see e.g. \cite{KR,Li,Li2,schoen_E_3,SS1,Sh_ineq,s_mixed,s_sumsets} and others.
In \cite{s_sumsets} the author proved that the bound $|A+B| \gg |A|^{3/2+c}$, $c>0$ takes place for wide class of {\it different} sets, having roughly comparable sizes.
For example, such bound holds if $A$ and $B$ have small multiplicative doubling.
It turns out that if (\ref{f:main_balog1_intr}) cannot be improved then there is some large set $C$ such that $|AC| \ll |A|$.
This allows us to apply results from \cite{s_sumsets}.

The author is grateful to  Tomasz Schoen for useful discussions.
%Sergey Konyagin and Misha Rudnev for useful discussions
%and, especially, Tomasz Schoen for very useful and fruitful explanations and  discussions.

%\section{Notation and required results}
\section{Notation}
\label{sec:definitions}

%We conclude with few comments regarding the notation used in this paper.
%Finally, with a slight abuse of notation
Let $\Gr$ be an abelian group and $+$ be the group operation.
In the paper we use the same letter to denote a set $S\subseteq \Gr$
and its characteristic function $S:\Gr \rightarrow \{0,1\}.$
By $|S|$ denote the cardinality of $S$.

Let $f,g : \Gr \to \C$ be two functions with finite supports.
Put
\begin{equation}\label{f:convolutions}
    (f*g) (x) := \sum_{y\in \Gr} f(y) g(x-y) \quad \mbox{ and } \quad
        %(f\circ g) (x) := \sum_{y\in \Gr} \ov{f(y)} g(y+x)
        (f\circ g) (x) := \sum_{y\in \Gr} f(y) g(y+x) \,.
\end{equation}

Let $A \subseteq \Gr$ be a set.
For any real $\a>0$ put
\begin{equation}\label{f:E_k_preliminalies_B}
    \E^{+}_\a (A)=\sum_{x\in \Gr} (A \c A)^{\a} (x)
\end{equation}
be {\it the higher energy}  of $A$.
In particular case $k=2$
%we put
we write
%$\E(A,B):= \E_2 (A,B)$ and
$\E^{+} (A) = \E^{+}_2 (A)$ and $\E(A,B)$ for $\sum_{x\in \Gr} (A \c A) (x) (B \c B) (x)$.
The quantity $\E^{+} (A)$ is called {\it the additive energy} of a set, see e.g. \cite{TV}.
For a sequence $s=(s_1,\dots, s_{k-1})$ put
$A^{+}_s = A \cap (A-s_1)\dots \cap (A-s_{k-1}).$
Then
$$
   \E^{+}_k (A) =  \sum_{s_1,\dots,s_{k-1} \in \Gr} |A^{+}_s|^2 \,.
$$
If we
%consider
have a
group $\Gr$ with a multiplication instead of addition  then we use symbol $\E^{\times}_\a (A)$ for the correspondent energy of a set $A$ and write $A^{\times}_s$ for $A \cap (As^{-1}_1)\dots \cap (As^{-1}_{k-1}).$
In the case of a unique operation we write just $\E_k (A)$, $\E(A)$ and $A_s$.

Let $A,B\subseteq \Gr$ be two finite sets. {\it The magnification ratio}
$R_B [A]$ of the pair $(A,B)$ (see e.g. \cite{TV}) is defined by
\begin{equation}\label{f:R_B[A]}
    R_B [A] = \min_{\emptyset \neq Z \subseteq A} \frac{|B+Z|}{|Z|} \,.
\end{equation}
A beautiful result on magnification ratio was proven by Petridis \cite{p}.

\begin{theorem}
    For any $A,B,C \subseteq \Gr$, we have
    \begin{equation}\label{f:Petridis_C}
        |B+C+X| \le R_{B} [A] \cdot |C+X| \,,
    \end{equation}
    where $X\subseteq A$ and $|B+X| = R_B [A] |X|$.
\label{t:Petridis_C}
\end{theorem}

\bigskip

We conclude the section by Ruzsa's triangle inequality, see e.g. \cite{TV}.
Interestingly, that our proof (developing some ideas of papers \cite{SS1}, \cite{B_RN_S})
does not require any mapping as usual.

\begin{lemma}
    Let $A,B,C\subseteq \Gr$ be any sets.
    Then
\begin{equation}\label{f:triangle_my}
    |C| |A-B| \le |A\times B - \Delta (C)| \le |A-C| |B-C| \,.
\end{equation}
\label{l:triangle_my}
\end{lemma}
\begin{proof}
We have
$$
    |A\times B - \Delta (C)| = \sum_{z\in A-B} |B \cap (A-z) - C| \ge |A-B| |C| \,.
$$
The inequality above is trivial and the identity follows by the projection of points $(x,y) \in A\times B - \Delta (C)$, $(x,y) = (a-c,b-c)$,
$a\in A$, $b\in B$, $c\in C$ onto $z:=x-y=a-b \in A-B$.
This concludes the proof.
$\hfill\Box$
\end{proof}

\bigskip

All logarithms are base $2.$ Signs $\ll$ and $\gg$ are the usual Vinogradov's symbols.

\section{Preliminaries}
\label{sec:preliminaries}

%We begin with a rather general definition of families of sets which are usually  obtained by Szemer\'{e}di--Trotter's
%theorem, see \cite{TV}.

As we discussed in the introduction our proof uses some notions from \cite{s_sumsets}.
So, let us recall the main definition of the paper.

\begin{definition}
    A set $A\subset \Gr$
    %is called
    has
    {\bf SzT--type}
    (in other words $A$ is  called {\bf Szemer\'{e}di--Trotter set})
    with parameter $\a \ge 1$
    if for any set $B\subset \Gr$ and an arbitrary $\tau \ge 1$ one has
\begin{equation}\label{f:SzT-type}
    |\{ x\in A+B ~:~ (A * B)(x) \ge \tau \}| \ll c (A) |B|^\a \cdot \tau^{-3} \,,
\end{equation}
    where $c (A)>0$ is a constant depends on the set $A$ only.
%    We define the quantity $c (A) |B|^\a$ as $c(A,B)$.
\label{def:SzT-type}
\end{definition}

%From the definition one can see that if $A$ has SzT--type then $(-A)$ has the same SzT--type with the same parameters %$\a$ and $c(A)$.

Simple calculations (or see \cite{s_sumsets}, Lemma 7) give us some connections between various energies of SzT--type sets.
Formula (\ref{f:Li}) below
%The following result
is due to Li \cite{Li}.

\begin{lemma}
    Suppose that $A,B,C\subseteq \Gr$ have SzT--type with the same parameter $\alpha$.
    Then
\begin{equation}\label{f:Li}
    \E^{3} (A) \ll \E^2_{3/2} (A) c(A) |A|^{\alpha} \,,
\end{equation}
\begin{equation}\label{f:Li_E}
    \E (A) \ll c^{1/2} (A) |A|^{1+\a/2} \,,
\end{equation}
    and
$$
    \sum_x (A\c A)(x) (B\c B) (x) (C\c C) (x) \ll (c(A) c(B) c(C))^{1/3} (|A| |B| |C|)^{\alpha/3}
        \times
$$
\begin{equation}\label{f:Li'}
        \times \log (\min\{ |A|, |B|, |C| \}) \,.
\end{equation}
\label{l:Li}
\end{lemma}

We need in Lemma 27 from \cite{SS1}.

\begin{lemma}
    Any set $A\subset \R$, $\R = (\R,+)$ has SzT--type with $\a=2$ and
$
    c(A) = |A| d(A)
$,
    where
\begin{equation}\label{f:d(A)}
    d(A) := \min_{C \neq \emptyset} \frac{|AC|^2}{|A| |C|} \,.
\end{equation}
\label{l:d(A)}
\end{lemma}

So, any set with small multiplicative doubling or, more precisely, with small quantity  (\ref{f:d(A)}) has SzT--type, relatively to addition,  in an effective way.
It can be checked that minimum in (\ref{f:d(A)}) is
actually  attained and we left the fact to an interested  reader.
%We left to the reader check the fact that minimum in (\ref{f:d(A)}) is  actually  attained.
Careful analysis of our proof gives that we do not need this.
%One can check the fact that minimum in (\ref{f:d(A)}) is actually attained.
%
%
Another examples of SzT--types sets can be found in \cite{s_sumsets}.

\bigskip

Now let us prove a simple result on $d(A)$, which follows from Petridis's Theorem \ref{t:Petridis_C}.
%wonderful arguments \cite{p}.

\begin{lemma}
    Let $A\subseteq \R^{+}$ be a set.
    Then $d(A) = d(A^{-1})$ and
\begin{equation}\label{f:d(AA)}
    d(AA) \le  \frac{|A|^2 d^2 (A)}{|AA||C|} \,,\quad \quad  d(A:A) \le \frac{|A|^2 d^2 (A)}{|A:A||C|} \,,
\end{equation}
    where $C$ is a set
    %such
    where
    the minimum in (\ref{f:d(A)}) is attained.
\label{l:d(AA)}
\end{lemma}
\begin{proof}
    The identity $d(A) = d(A^{-1})$ is obvious.
    Let us prove (\ref{f:d(AA)}).
    Suppose that the minimum in (\ref{f:d(A)}) is attained at $C$.
    %By Lemma \ref{l:plunnecke-ruzsa}
    By Theorem \ref{t:Petridis_C} there is $X\subseteq C$
    %, $|X| \ge |C|/2$
    such that $|AAX| \le R |AX|$, where $R = R_A [C]$ is defined by formula (\ref{f:R_B[A]}).
    %Thus
    We have
\begin{equation}\label{tmp:24.01.2015_1}
    d(AA) \le \frac{|AAX|^2}{|AA||X|} \le R^2 \frac{|AX|^2}{|AA||X|} = \frac{|AX|^4}{|AA||X|^3}
        \le
            \frac{|AC|^4}{|AA||C|^3}
        =
            \frac{d^2 (A) |A|^2}{|AA| |C|} \,.
\end{equation}
    Similarly, we take $C$ such that the correspondent minimum for $d(A^{-1})$ is attained at $C$.
    Further, let $Y \subseteq C$ is given by Theorem \ref{t:Petridis_C}
    %, $|Y| \ge |C|/2$ and put
    and put $R=R_{A} [C^{-1}]$.
%    $$
%        R=\frac{|AY|}{|Y|} = \min_{\emptyset \neq Y\subseteq C^{-1}} \frac{|AY|}{|Y|} \le \frac{|A^{-1} C|}{|C|} \,.
%    $$
    Then $|(A:A) Y| \le R |A^{-1} Y| \le R|A^{-1} C|$, $R= |AY^{-1}|/|Y| \le |A C^{-1}|/|C|$
    and
    arguments similar to (\ref{tmp:24.01.2015_1}) can be applied.
    This completes the proof.
$\hfill\Box$
\end{proof}

\begin{remark}
    Actually, the proof of Lemma \ref{l:d(AA)} gives us $d(AA) \le \frac{|AC|^4}{|AA||C|^3}$,
    $d(A:A) \le \frac{|AC|^4}{|A:A||C|^3}$ for any nonempty $C$.
\label{r:l_d(AA)}
\end{remark}

%The result below it is a full version of Theorem \ref{t:Solymosi_intr} from introduction.

\bigskip

Finally, we
%give
formulate
a full version of Theorem \ref{t:Solymosi_intr} from the introduction.

\begin{theorem}
    Let $A,B\subseteq \R$ be sets, $\tau>0$ be a real number.
    Then
\begin{equation}\label{f:Solymosi}
%    |S_\tau (A,B)| :=
    |\{ x ~:~ |A \cap xB| \ge \tau \}| \ll \frac{|A+A| |B+B|}{\tau^2} \,.
\end{equation}
    In particular
\begin{equation}\label{f:Solymosi_E}
    \E^{\times} (A,B) \ll |A+A| |B+B| \cdot \log (\min\{|A|, |B| \}) \,.
\end{equation}
\label{t:Solymosi}
\end{theorem} 

%\section{The proof}
\section{The proof of the main results}
\label{sec:proof}

Our proof relies  on a  partial case of Theorem 14 from \cite{s_sumsets}.

\begin{theorem}
    Suppose that $A,A_*\subset \R$ have SzT--type with the same parameter $\a=2$.
    Then
$$
    |A \pm A_*| \gg
        \max\{
        d (A_*)^{-\frac{1}{3}} d(A)^{-\frac{2}{9}} |A_*|^{\frac{8}{9}} |A|^{\frac{2}{3}},
        d (A)^{-\frac{1}{3}} d(A_*)^{-\frac{2}{9}} |A|^{\frac{8}{9}} |A_*|^{\frac{2}{3}} \,,
$$
\begin{equation}\label{f:p_main_diff_1}
        \min\{
        d (A_*)^{-\frac{2}{27}} d(A)^{-\frac{13}{27}} |A_*|^{\frac{14}{9}},
        d (A)^{-\frac{2}{27}} d(A_*)^{-\frac{13}{27}} |A|^{\frac{14}{9}} \} \}
            \times
        (\log (|A| |A_*|))^{-\frac{2}{9}} \,.
\end{equation}
\label{t:main_diff}
\end{theorem}

Now we can prove the main result of the paper.

\begin{theorem}
    Let $A\subset \R$ be a finite set.
    Then
    %there are $\eps_1>0$, $\eps_2>0$ such that
\begin{equation}\label{f:main_balog1}
    | A:A +  A | \gg |A|^{\frac{3}{2}+\frac{1}{82}} \cdot (\log |A|)^{-\frac{2}{41}} \,,
\end{equation}
    and
\begin{equation}\label{f:main_balog2}
    |AA + A| \gg
    %\max\{
    |AA|^{\frac{11}{41}} |A:A|^{-\frac{11}{41}} |A|^{\frac{62}{41}} (\log |A|)^{-\frac{2}{41}}
    %,
    %    |A|^{\frac{5}{2}} |AA|^{-\frac{3}{4}}, |A|^{\frac{5}{2}} |A:A|^{-\frac{3}{4}} \}
    \,.
\end{equation}
%    Moreover
%\begin{equation}\label{f:AA+A2}
%    \max\{ |A:A+A|, |AA+A| \} \gg |A|^{\frac{3}{2}+\frac{1}{82}} \cdot (\log |A|)^{-\frac{2}{41}} \,.
%\end{equation}
\label{t:main_balog}
\end{theorem}
\begin{proof}
    Put $l=\log |A|$.
    Without loosing of generality, we can assume that $0\notin A$.
    Suppose that $|A:A + A| \ll M|A|^{3/2}$, $|AA + A| \ll M|A|^{3/2}$, where $M$ is a small power of $|A|$, that is $M=|A|^\eps$ and
    %we will
    obtain a contradiction.
    Let us begin with
    %(\ref{f:main_balog1}) and
    (\ref{f:main_balog1})
    because the proof of the second inequality requires some additional steps.
%%    The proof of the first and the second inequalities requires some additional
    %arguments.
%%    steps.
%    Suppose that $|A:A +A| \ll M|A|^{3/2}$, where $M$ is small power of $|A|^\eps$ and we will obtain a contradiction.

    Recall the arguments from \cite{balog} or see the proof of Theorem 31 from \cite{SS1}.
    %We will closely follow Balog's proof, so we only sketch the argument.
    Let $l_i$ be the line
$y=q_ix.$ Thus, $(x,y)\in l_i\cap A^2$ if and only if $x\in  A^\m_q.$
Let $q_1,\dots,q_n\in \Pi \subseteq A/A$ be such  that $q_1<q_2<\dots<q_n$.
Here $\Pi$ is a set which can vary, in principle, and at the moment we choose $\Pi$ such that $|A^\m_{q_i}|\ge 2^{-1} |A|^2/|A/A|$ for all $q_i \in \Pi$.
% :=D,$
%so that
Thus,
$\sum_{i\in \Pi} |A^\m_{q_i}|\ge \frac12|A|^2$.
%, and hence $D|\Pi| \gg |A|^2$.
 We multiply all points of $A^2$ lying on the line $l_i$ by
$\D(A^{-1})$, so we obtain $|A^\m_{q_i}:A|$ points still belonging to the line $l_i$ and then we consider sumset of
the resulting set with $l_{i+1}\cap A^2.$
Clearly, we
%obtain
get
$|A^\m_{q_i}:A||A^\m_{q_{i+1}}|$ points from the set $(A:A+A)^2$
lying between the lines $l_i$ and $l_{i+1}$.
Therefore, using the definition of the number $d(A)$, we have
\begin{equation}\label{tmp:21.01.2015_1}
    M^2 |A|^3 \gg |A:A + A|^2 \ge \sum_{i=1}^{n-1} | A^\times_{q_i}||A^\times_{q_{i+1}}:A|
    \ge
        |A|^{1/2} d^{1/2} (A) \sum_{i=1}^{n-1} | A^\times_{q_i}|^{3/2}
            \gg
\end{equation}
\begin{equation}\label{tmp:21.01.2015_2}
            \gg
                |A|^{3/2} d^{1/2} (A) |A:A|^{-1/2} \sum_{i=1}^{n-1} | A^\times_{q_i}|
                    \gg
                        |A|^{7/2} d^{1/2} (A) |A : A|^{-1/2} \,.
                        %\min_{q\in \Pi} |A^\times_{q}:A| \,.
%    \gg \frac{|A|^2}{|A/A|}\sum_{i=1}^{n-1} |AA^\times_{q_{i+1}}|\,,
\end{equation}
%    Thus, we see that there is $C=A^\times_q$ such that $|C| \ge D$ and  $|A/C| \ll M^2 |A|$.
%    It means that
    Thus,
    \begin{equation}\label{f:d(A)_recall}
        d(A) = \min_{C \neq \emptyset} \frac{|AC|^2}{|A||C|}
            \ll \frac{M^4|A/A|}{|A|} \,.
    \end{equation}
    To estimate $d(A:A)$, $d(AA)$ we
    %apply
    use
    %(arguments of)
    Lemma \ref{l:d(AA)}, see Remark \ref{r:l_d(AA)}.
    In other words, our $C$ is some $A^\times_{q_i}$, where $q_i \in \Pi$.
    After that applying  the first inequality of Theorem \ref{t:main_diff} with $A=A$, $A_* = A:A$
    , we obtain
%    and obtain
$$
    M |A|^{3/2}
        \ge
    |A:A + A| \gg |A:A|^{8/9} |A|^{2/3} d^{-2/9} (A) \left( \frac{|A|^2 d^2 (A)}{|A:A||C|} \right)^{-1/3} l^{-2/9}
        =
$$
$$
        =
        |A:A|^{11/9} d^{-8/9} (A) |C|^{1/3} l^{-2/9}
            \gg
                |A|^{14/9} M^{-32/9} l^{-2/9} \,,
$$
and hence $M \gg l^{-2/41} |A|^{1/82}$.
This implies (\ref{f:main_balog1}).

It remains to prove (\ref{f:main_balog2}).
In the case we multiply all points of $A^2$ lying on the line $l_i$ by
$\D(A^{})$, so we obtain $|A A^\m_{q_i}|$ points still belonging to the line $l_i$ and then we consider sumset of
the resulting set with $l_{i+1}\cap A^2.$ Clearly, we obtain $|A A^\m_{q_i}||A^\m_{q_{i+1}}|$ points from the set $(AA+A)^2$.
Thus,
\begin{equation}\label{tmp:21.01.2015_1+}
     M^2 |A|^3 \gg |AA + A|^2 \ge \sum_{i=1}^{n-1} | A^\times_{q_i}||AA^\times_{q_{i+1}}|
\end{equation}
and we repeat the arguments above.
The proof  gives us
\begin{equation}\label{tmp:26.01.2015_1}
    |AA+A| \gg |AA|^{11/41} |A|^{-4/41} (\E^\times_{3/2} (A))^{22/41} l^{-2/41} \,.
\end{equation}
Here we have chosen the set $\Pi$ as $\sum_{q\in \Pi} |A^\times_q|^{3/2} \gg \E^\times_{3/2} (A)$
or, in other words, $|A^\times_q| \gg (\E^\times_{3/2} (A))^2 |A|^{-4}$.
Using the H\"{o}lder inequality, combining with (\ref{tmp:26.01.2015_1}), we get
$$
    |AA + A| \gg |AA|^{\frac{11}{41}} |A:A|^{-\frac{11}{41}} |A|^{\frac{62}{41}} l^{-2/41} \,.
$$
%Finally, applying Lemma \ref{l:Li} with $\a=2$, $c(A) = |A| d(A)$, we obtain
%$$
%    |AA + A|^4 \gg d(A) |A| (\E^{\times}_{3/2} (A))^2
%        \gg
%            \frac{(\E^{\times} (A))^3}{|A|^2}
%                \ge
%                    \frac{|A|^{10}}{|AA|^3}
%$$
%and similar for $A:A$.
This completes the proof.
$\hfill\Box$
\end{proof}

\begin{remark}
    Using the full power of Theorem 14 from \cite{s_sumsets} one can obtain further results connecting $|AA:A|$, $|A:AA|$ with $|AA + A|$, $|A:A + A|$ and so on. We do not make such calculations.
\end{remark}

The same method allows us to improve the result of Balog concerning the size of $AA+AA$ and $A:A+A:A$.

\begin{theorem}
    Let $A\subset \R$ be a set.
    Then
    %there are $\eps_1>0$, $\eps_2>0$ such that
\begin{equation}\label{f:main_balog1'}
    | A:A +  A:A | \gg
    %\max\{
    |A:A|^{\frac{14}{29}} |A|^{\frac{30}{29}} (\log |A|)^{-\frac{2}{29}}
    %, \E^{\times} (A) |A|^{-1} (\log |A|)^{-1}
        %\}
        \,,
\end{equation}
    and
\begin{equation}\label{f:main_balog2'}
    |AA + AA| \gg
    %\max\{
    |AA|^{\frac{19}{29}} |A:A|^{-\frac{5}{29}} |A|^{\frac{30}{29}} (\log |A|)^{-\frac{2}{29}}
    %, \E^{\times} (A) |A|^{-1} (\log |A|)^{-1} \}
        \,.
\end{equation}
%    Moreover
\label{t:AA+AA}
\end{theorem}
\begin{proof}
    As in the proof of Theorem \ref{t:main_balog}, we define $l_i$ to be the line $y=q_ix$
%    Thus, $(x,y)\in l_i\cap A^2$ if and only if $x\in  A^\m_q.$
and
$q_1,\dots,q_n\in \Pi \subseteq A/A$ be such  that $q_1<q_2<\dots<q_n$ and $|A^\m_{q_i}|\ge 2^{-1} |A|^2/|A/A|$
for any $q_i \in \Pi$.
% :=D,$
Thus, $\sum_i |A^\m_{q_i}|\ge \frac12|A|^2$.
%, and hence $D|\Pi| \gg |A|^2$.
 We multiply all points of $A^2$ lying on all lines $l_i$ by
$\D(A^{-1})$, so we obtain $|A^\m_{q_i}:A|$ points still belonging to the line $l_i$ and then we consider sumset of
the resulting set with itself.
Clearly, we
%obtain
get
$|A^\m_{q_i}:A||A^\m_{q_{i+1}}:A|$ points from the set $(A:A+A:A)^2$
lying between the lines $l_i$ and $l_{i+1}$. Therefore, we have
$$
     %M^2 |A|^3 \gg
     \sigma^2 :=
        |A:A + A:A|^2 \ge \sum_{i=1}^{n-1} | A^\times_{q_i}:A||A^\times_{q_{i+1}}:A|
        \ge
$$
\begin{equation}\label{tmp:22.01.2015_1}
        \ge
            d(A) |A| \sum_{i=1}^{n-1} | A^\times_{q_i}|^{1/2} | A^\times_{q_{i+1}}|^{1/2}
                \gg
                    |A|^3 d(A) \,.
\end{equation}
It gives $d(A) \ll \sigma^2 |A|^{-3}$.
Using Theorem \ref{t:main_diff} with $A=A_* = A/A$, we obtain
\begin{equation}\label{tmp:27.01.2015_1}
    \sigma \gg |A/A|^{14/9} \left( \frac{|A|^2 d^2 (A)}{|A/A||C|} \right)^{-5/9} l^{-2/9}
        \gg
            |A/A|^{19/9} |A|^{-10/9} |C|^{5/9} d^{-10/9} (A)  l^{-2/9}
                \gg
\end{equation}
\begin{equation}\label{tmp:27.01.2015_2}
                \gg
                    |A/A|^{14/9} \sigma^{-20/9} |A|^{10/3} l^{-2/9} \,.
\end{equation}
After some calculations, we
%obtain
get
$\sigma \gg |A/A|^{14/29} |A|^{30/29} l^{-2/29}$.

To obtain (\ref{f:main_balog2'}) we use  the previous arguments
%and the second part
$$
     \sigma^2 := |AA + AA|^2 \ge \sum_{i\in \Pi} | AA^\times_{q_i}| |AA^\times_{q_{i+1}}|
        \ge
            d(A) |A| \sum_{i\in \Pi} | A^\times_{q_i}|^{1/2} | A^\times_{q_{i+1}}|^{1/2}
                \gg
$$
\begin{equation}\label{tmp:22.01.2015_1+}
    \gg
                    d(A) |A| |\Pi| \Delta \,,
\end{equation}
choosing $\Pi\subseteq A/A$ such that for any $q\in \Pi$ one has
%$\Delta < |A^{\times}_q| \le 2\Delta$
$|A|^2 / |A:A| \ll \Delta \le |A^{\times}_q|$.
%and $|\Pi| \D^{3/2} \gg \E^{\times}_{3/2} (A)$.
Clearly, such set $\Pi$ exists by
%the pigeonhole principle.
simple average arguments.
The
%arguments
calculations
as in (\ref{tmp:27.01.2015_1})---(\ref{tmp:27.01.2015_2}) give us
\begin{equation*}\label{tmp:27.01.2015_1'}
    \sigma \gg |AA|^{14/9} \left( \frac{|A|^2 d^2 (A)}{|AA|\D} \right)^{-5/9} l^{-2/9}
        \gg
            |AA|^{19/9} |A|^{-10/9} \D^{5/9} d^{-10/9} (A)  l^{-2/9}
                \gg
\end{equation*}
\begin{equation*}\label{tmp:27.01.2015_2'}
                \gg
                    |AA|^{19/9} \sigma^{-20/9} (|\Pi| \D^{3/2})^{10/9} l^{-2/9} \,.
\end{equation*}
%Applying H\"{o}lder inequality, we get
After some computations, we
%get
obtain
$$
    \sigma \gg |AA|^{19/29} |A:A|^{-5/29} |A|^{30/29} l^{-2/29} \,.
$$
%
%
%
%Finally, combining  Lemma \ref{l:Li} with (\ref{tmp:22.01.2015_1+}), we get for any set $\Pi$
%and the correspondent number $\D$ that
%$$
%    \sigma^2 \gg \frac{(\E^{\times} (A))^3 |\Pi| \Delta}{|A|^2 (\E^\times_{3/2} (A))^2} \,.
%$$
%Now choose $\Pi$ and $\D$ such that $|\Pi| \D^{3/2} \gg \E^\times_{3/2} (A) l^{-1}$
%and the size of $A^\times_q$ differ at most twice on $\Pi$.
%Clearly, such a set exists by the pigeonhole principle.
%Note also that, trivially, $|\Pi| \D^2 \le \E^\times (A)$.
%Using this, we have
%$$
%    \sigma^2 \gg \frac{(\E^{\times} (A))^3}{|A|^2 |\Pi| \D^2 l^2}
%        \gg
%            \frac{(\E^{\times} (A))^2}{|A|^2 l^2}
%$$
%as required.
%Obviously, the second estimate in (\ref{f:main_balog1'}) can be obtained in a similar way.
This concludes the proof.
$\hfill\Box$
\end{proof}

%\bigskip

%\begin{corollary}
%    Let $A$ be a finite set of reals. Then
%\begin{equation}\label{f:AA+AA_1.6}
%    |AA+AA|\,, \quad |A:A+A:A| \gg |A|^{\frac{5}{3}} \log^{-\frac{1}{3}} |A| \,.
%\end{equation}
%\label{c:AA+AA_1.6}
%\end{corollary}
%\begin{proof}
%    Let us consider $AA+AA$, the proof for $A:A+A:A$ is similar.
%    On the one hand, by Theorem \ref{t:AA+AA}
%\begin{equation}\label{f:Em_new}
%    \E^{\times} (A) \ll |A| |AA+AA| \log |A| \,.
%\end{equation}
%    Thus, using the Cauchy--Schwarz inequality, we have
%\begin{equation}\label{tmp:27.01.2015_3}
%    |AA+AA| \gg |A|^3 |A:A|^{-1} \log^{-1} |A| \,.
%\end{equation}
%    On the other hand, applying formula (\ref{f:Balog2}) of Theorem \ref{t:Balog}, we obtain
%\begin{equation}\label{tmp:27.01.2015_4}
%    |AA+AA| \gg |A| |A:A|^{1/2} \,.
%\end{equation}
%    Combining bounds (\ref{tmp:27.01.2015_3}), (\ref{tmp:27.01.2015_4}) and optimizing over $|A:A|$, we get  %(\ref{f:AA+AA_1.6}).
%$\hfill\Box$
%\end{proof}

\bigskip

Finally, let us obtain a result on $AA+A$, $AA+AA$ of another type.

\begin{proposition}
    Let $A\subset \R$ be a set.
    Then
\begin{equation}\label{f:1}
    |AA+A|^4 \,,~ |A:A+A|^4 \gg |A|^{-2} (\E^\times_{3/2} (A))^2 \E^{+}_3 (A) \log^{-1} |A| \,,
\end{equation}
    and
\begin{equation}\label{f:2}
    |AA+AA|^2 \,,~ |A:A+A:A|^2 \gg \E^{+}_3 (A) \log^{-1} |A| \,.
\end{equation}
%    In particular
    Moreover
\begin{equation}\label{f:1'}
    |AA+A|^4 \,,~ |A:A+A|^4 \gg \frac{|A|^{10}}{|A:A| |A-A|^2} \,,
\end{equation}
\begin{equation}\label{f:2'}
    |AA+AA|^2 \,,~ |A:A+A:A|^2 \gg \frac{|A|^6}{|A-A|^2} \,.
\end{equation}
\label{p:AA+AA_Solymosi}
\end{proposition}
\begin{proof}
Put $l=\log |A|$.
Using Lemma \ref{l:Li}, we obtain
%\ref{l:d(A)} (or see calculations from \cite{s_sumsets}), we obtain
for any $A,B$ and $C$
%that
\begin{equation}\label{tmp:27.01.2015_5}
    \sum_x (A\c A)(x) (B\c B) (x) (C\c C) (x) \ll |A| |B| |C| (d(A) d(B) d(C))^{1/3} \log (|A||B||C|) \,.
\end{equation}
In particular case $A=B=C$ of formula above the definition of the number $d(A)$
%it means that , one has
gives us
\begin{equation}\label{tmp:22.01.2015_2}
    |A A^\times_s|^2 \,,~ |A : A^\times_s|^2 \gg |A|^{-2} |A^{\times}_s| \E^{+}_3 (A) l^{-1}
    %\,.
\end{equation}
for any $s\in A:A$.
Applying (\ref{tmp:21.01.2015_1}), (\ref{tmp:21.01.2015_1+}) and the last bound, we obtain (\ref{f:1}).
Using (\ref{tmp:22.01.2015_2}) one more time and Katz--Koester inclusion \cite{kk}, namely,
\begin{equation}\label{f:kk_m}
    AA^{\times}_s \subseteq AA \cap s AA \,, \quad A:A^{\times}_s \subseteq (A:A) \cap s^{-1} (A:A)
    %\,.
\end{equation}
as well as formula (\ref{f:Solymosi}) of Solymosi's result, we get (\ref{f:2}).
%Applying (\ref{tmp:22.01.2015_1}), (\ref{tmp:22.01.2015_1+}) and the last bound, we obtain (\ref{f:2}).
Another way to prove (\ref{f:2}) is just to use formulas  (\ref{tmp:22.01.2015_1}), (\ref{tmp:22.01.2015_1+}),
combining with (\ref{tmp:22.01.2015_2}).

%Formulas
Inequalities
(\ref{f:1'}), (\ref{f:2'}) follow similar to (\ref{f:1}), (\ref{f:2})
%by the Cauchy--Schwarz inequality and a little bit more accurate calculations.
from a direct application of Definition \ref{def:SzT-type} and the H\"{o}lder inequality.
% or just by the H\"{o}lder inequality.
This completes the proof.
$\hfill\Box$
\end{proof}

\begin{remark}
    %From estimate (\ref{f:Em_new}) it follows that $\E^\times (A) \ll |A| |AA+AA| \log |A|$.
    Applying arguments as in the proof (\ref{f:2})
    %and
    as well as
    formula (\ref{f:Li_E}) of Lemma \ref{l:Li},
    %as well as the Cauchy--Schwarz inequality,
    we obtain a similar bound, namely,
    $$\E^{+} (A) \ll |A| |AA+AA|$$
    (actually, using methods from \cite{Sh_ineq} one can improve the
    %estimate).
    inequality).
    It is interesting to compare this estimate with Solymosi's upper bound  for the multiplicative energy  (\ref{f:Solymosi_E}).
    Using formula (\ref{f:Li}) of Lemma \ref{l:Li}, we have also
    $$
        (\E^{+} (A))^{3/2} \E^{\times}_{3/2} (A) \ll \E^{+}_{3/2} (A) |A| |AA+A|^2 \,.
    $$
\label{r:E^+}
\end{remark}

\bigskip

Combining inequality  (\ref{f:2}) with some estimates from \cite{Sh_energy}, we obtain a result in spirit of paper \cite{RN_Z}.

\begin{corollary}
    Let $A\subset \R$ be a set.
    Suppose that
\begin{equation}\label{cond:RN_Z}
    |(A+A)(A+A) + (A+A)(A+A)| \ll |A|^2 \quad \mbox{ and } \quad \E^{+} (A) |A-A| \ll |A|^4 \,.
\end{equation}
    Then
\begin{equation}\label{f:RN_Z}
    |A-A| \ll |A| \log^{4/7} |A| \,.
\end{equation}
    The same holds if one replace sum onto minus and product onto division into the first condition from (\ref{cond:RN_Z}).

    If just the first condition of (\ref{cond:RN_Z}) holds (with plus) then
\begin{equation}\label{f:RN_Z'}
    |A \pm A| \ll |A| \log^{} |A| \,,
\end{equation}
    and if it holds with minus then
\begin{equation}\label{f:RN_Z''}
    |A - A| \ll |A| \log^{} |A| \,,
\end{equation}
    Again, one can replace product onto division in the first condition of (\ref{cond:RN_Z}).
\label{c:RN_Z}
\end{corollary}
\begin{proof}
    Let us have deal with the
    %case
    situation of the sum and the product.
    Another cases can be considered similarly.
    By Theorem 30 from \cite{Sh_energy} and our second condition, one has
    $$
        \E^{+}_3 (A\pm A) \ge |A|^{45/4} |A-A|^{-1/2} (\E^{+} (A))^{-9/4}
            \gg
                |A|^{9/4} |A-A|^{7/4} \,.
    $$
    On the other hand, using formula (\ref{f:2}) from Proposition \ref{p:AA+AA_Solymosi} and
    our first condition, we get
    $$
        |A|^4 \log |A| \gg \E^{+}_3 (A\pm A) \gg |A|^{9/4} |A-A|^{7/4}
    $$
    as required.

    Finally, using the additive variant of Katz--Koester inclusion (\ref{f:kk_m}) (or see Proposition 29 from \cite{Sh_energy}),
%    and Remark \ref{r:E^+},
    we obtain
    $$
        |A|^3 |A \pm A| \le \E^{+}_3 (A + A) \ll |A|^4 \log |A| \,,
    $$
    and
    $$
        |A|^3 |A - A| \le \E^{+}_3 (A - A) \ll |A|^4 \log |A| \,.
    $$
    This completes the proof.
$\hfill\Box$
\end{proof}

\bigskip

\noindent{I.D.~Shkredov\\
Steklov Mathematical Institute,\\
ul. Gubkina, 8, Moscow, Russia, 119991}
%MSU, IPPI RAN\\}
%%\\
%and
%\\
%Delone Laboratory of Discrete and Computational Geometry,\\
%Yaroslavl State University,\\
%Sovetskaya str. 14, Yaroslavl, Russia, 150000
\\
and
\\
IITP RAS,  \\
Bolshoy Karetny per. 19, Moscow, Russia, 127994\\
{\tt ilya.shkredov@gmail.com}

\end{document}